\documentclass{amsart}

\usepackage{fancybox}
\usepackage{color}
\usepackage{leftidx}
\usepackage{amsmath,amssymb,graphicx,epsfig}
\usepackage{amsmath}
\usepackage{amssymb}

 \newtheorem{thm}{Theorem}[section]
 \newtheorem{cor}[thm]{Corollary}
\newtheorem{lem}[thm]{Lemma}
 \newtheorem{prop}[thm]{Proposition}

 \newtheorem{defn}[thm]{Definition}
 \newtheorem{rem}[thm]{Remark}

\newcommand{\eps}{\varepsilon}

\newcommand \be     {\begin{equation}}
\newcommand \ee     {\end{equation}}
\newcommand {\RR} {\mathbb{R}}
\newcommand {\Rplus} {\mathbb{R}_+}
\newcommand {\Rplusc} {\overline{\mathbb{R}_+}}

\newcommand \Rn    {\mathbb{R}^n}

\newcommand{\set}[1]{\left\{#1\right\}}

\newcommand \Gcal   {\mathcal G}

  \newcommand{\suml}{\sum\limits}
  
  \newcommand{\pa}{\partial}
 \numberwithin{equation}{section}

\begin{document}

\title [FLUXES IN BALANCE LAWS]
{REGULARITY OF FLUXES IN NONLINEAR HYPERBOLIC BALANCE LAWS }

\author{Matania Ben-Artzi}
\address{Matania Ben-Artzi: Institute of Mathematics, The Hebrew University, Jerusalem 91904, Israel}
\email{mbartzi@math.huji.ac.il}
\author{Jiequan Li}
\address{Jiequan Li: Laboratory of Computational Physics,  Institute of Applied Physics and Computational Mathematics,
 Beijing , China; Center for Applied Physics and Technology, Peking University, China;  and State Key Laboratory for Turbulence Research and Complex System, Peking University, China }
\email{li\_jiequan@iapcm.ac.cn}





\thanks{ The first author thanks the Institute of Applied Physics and Computational Mathematics, Beijing,  for the hospitality and support. The second author is supported by NSFC (nos. 11771054, 91852207), the Sino-German Research Group Project, National key project(GJXM92579) and Foundation of LCP. It is a pleasure to thank C. Dafermos and M. Slemrod for many useful comments. }




\date{\today}




\begin{abstract} This paper addresses the issue of the formulation of weak solutions to systems of nonlinear hyperbolic conservation laws as integral balance laws. The basic idea is that the ``meaningful objects'' are the \textit{fluxes}, evaluated across domain boundaries over time intervals. The fundamental result in this treatment is the regularity of the flux trace in the multi-dimensional setting. It implies that a weak solution indeed satisfies the balance law. In fact, it is shown that the flux is Lipschitz continuous with respect to suitable perturbations of the boundary.

\end{abstract}
\maketitle

\section{\textbf{INTRODUCTION}}

 This paper deals with the formulation of weak solutions of nonlinear hyperbolic conservation laws as solutions of integral ``balance laws''.  Such laws are closely associated with the relevant physical  laws. The basic idea is that the ``meaningful objects'' are the \textit{fluxes}, evaluated across manifolds over time intervals. In contrast, the role played by the unknown $u(x,t)$ is not its pointwise values, but is limited to its integral over a given domain as function of time. A fundamental issue is therefore the meaning (and regularity) of fluxes across domain boundaries.

  From the numerical point of view, finite volume schemes rely on an appropriate approximation of these fluxes, so the present paper is a contribution to the validity of the finite volume approach.

 The case of a single space dimension has already been studied by the authors in ~\cite{mathcomp}, and the emphasis here is on systems of conservation laws in the multi-dimensional case.

         Consider a system of hyperbolic conservation laws in $\Rn$ of the form
 \be\label{eqconslaw} u(x,t)_t+\nabla\cdot f(u(x,t))=0,\quad u=(u_1,\ldots,u_D)\in\RR^D,\quad (x,t)\in\Rn\times\RR_+,
 \ee
 where the matrix of fluxes is $$f(u)=(f_1(u),\ldots,f_D(u)),\,\,\,f_i(u)\in\Rn,\,\,i=1,2,\ldots,D,$$
     subject to initial data

    \be\label{eqinitdata}
    u(x,0)=u_0(x),\quad x\in\Rn.
    \ee

   Let $\Omega\subseteq\Rn$ be a bounded smooth domain with $\Gamma=\pa\Omega,$ and let $0\leq t_1<t_2.$
 Formally, by integration of the equation in $Q=\Omega\times[t_1,t_2]\subseteq \Rn\times\overline{\RR_+}$ the following ``balance'' equality holds.

    \be\label{eqbalancecons}
     \int_{\Omega}u_i(x,t_2)dx-\int_{\Omega}u_i(x,t_1)dx=
     -\Big[\int_{t_1}^{t_2}\int_{\Gamma}f_i(u(x,t))\cdot\nu dS\,dt\Big],\quad i=1,2,\ldots,D.
    \ee
 Here $\nu$ is the outward unit normal to $\Gamma$ and $dS$ is the surface Lebesgue measure.

 \textbf{NOTATION.} Let $X$ be a space of \textit{scalar} functions. Then we denote by $X\bigotimes\RR^D$ the space of vector functions of $D$ components, where each component is an element of $X.$ Thus $C^\infty_0(\Rn)\bigotimes\RR^D$ is the space of $D-$vectors whose components are test functions in $\Rn.$

  Equation ~\eqref{eqbalancecons} can be considered as an integrated (formal) form of ~\eqref{eqconslaw}, using the Gauss-Green theorem.  However, the application of this theorem is certainly not straightforward, since the function $u(x,t)$ is not even continuous (see ~\cite[Section 4.5]{federer}). We refer to ~\cite{chen, silhavy2} and ~\cite[Chapter I]{dafermos} for an abstract discussion of this topic. Regarding the right-hand side of ~\eqref{eqbalancecons} one needs to keep in mind the following comment concerning the identification of the boundary flux: ``the drawback of this, functional analytic, demonstration is that it does not provide any clues on how the $q_\mathfrak{D}$ may be computed from $A$'' ~\cite[Section 1.3]{dafermos}.

In the context of theoretical continuum mechanics the quantity $\int_{A}f(u(x,t))\cdot\nu dS, \,\,A\subseteq\Gamma,$ is referred to as the \textbf{Cauchy flux across $A$} ~\cite{gurtin,silhavy1}. The pointwise value $f(u(x,t))\cdot\nu$ is its density.

We now introduce the notion of a ``solution to the balance law'' as follows.

\begin{defn}\label{defnbalance}  Let $$u_0\in L^1(\Rn)\cap L^\infty(\Rn)\bigotimes\RR^D.$$
 The function $u(\cdot,t)\in C(\Rplusc,L^1(\Rn))\cap L^\infty(\Rplus,L^\infty(\Rn))\bigotimes\RR^D$ is a solution to the balance law ~\eqref{eqbalancecons} corresponding to the partial differential equation  ~\eqref{eqconslaw}  if the following  conditions are satisfied.
    \begin{itemize}
    \item  For every $t\geq 0$ and every smooth bounded domain $\Omega\subseteq\Rn$  the integral $\int_{\Omega}u(x,t)dx$ is well defined and is a continuous function of $t.$
    \item  For every smooth bounded domain $\Omega\subseteq\Rn$ and interval  $[t_1,t_2]\subseteq \overline{\RR_+}$ the trace $$h_i(t_1,t_2)=\int_{t_1}^{t_2}\Big[\int_{\pa\Omega}f_i(u(x,t))\cdot\,\nu\, dS_x\Big]dt,\quad i=1,2,\ldots,D,$$
      is well defined, and is  continuous with respect to suitable perturbations of the boundary $\pa\Omega$ (see Lemma ~\ref{lemtrace} below for details).
      We denoted by $dS_x$ the Lebesgue surface measure on $\pa\Omega.$
     \item The balance equation ~\eqref{eqbalancecons} is satisfied.





    \end{itemize}
    \end{defn}
    \begin{defn}\label{defflux}
    The quantities  $h_i(t_1,t_2),\quad i=1,2,\ldots,D,$ are called the \textbf{fluxes} associated with the conservation law  ~\eqref{eqconslaw}, across the boundary $\pa\Omega$ over the time interval $[t_1,t_2].$
    \end{defn}
    \begin{rem}\label{rembalancelaw} Our definition of a solution to the balance law conforms to that introduced in ~\cite[Chapter I]{dafermos}. In fact, in Dafermos' book the balance equation is assumed to hold for any domain in spacetime. We note that other authors use various other terms, such as the ``integral conservation law'', and the term ``balance law'' is applied to a conservation law with a source term.
    \end{rem}


Definition ~\ref{defnbalance} is closely related to the physical interpretation of fluid mechanics under the continuum hypothesis, that stipulates that the total quantities $\int_{\Omega} u(x,t)dx$ in a fixed domain are  well-defined and continuous in time. The   fluxes $h_i(t_1,t_2)$ are defined over a time interval $[t_1,t_2]$ rather than at any instant $t,$ reflecting  the dynamical process of fluid flows. 

      The important concept of \textbf{divergence-measure vectorfields} was introduced in ~\cite[Definition 1.2]{chen-frid} and then generalized in ~\cite{chen-comi}. In particular if $u(x,t)$ is a weak solution to ~\eqref{eqconslaw} then $(u(x,t),f(u(x,t)) $ is divergence-measure in spacetime. For such vectorfields the Gauss-Green equation can be justified ~\cite[Theorem 2.2]{chen-frid}, provided the domain boundary is a deformable Lipschitz boundary. The resulting  flux turns out to be a Radon measure on the boundary. In our treatment here we treat more specifically weak solutions to the conservation law ~\eqref{eqconslaw}. A distinction is made between the time coordinate and the spatial coordinates. Thus, the vectorfield $f(u(x,t))$ is shown (Subsection ~\ref{subsecsystemlaw}) to be divergence-measure in space, for a.e. time $t.$ However, we do not invoke the Gauss-Green formula at \textit{fixed time levels}, but show (Subsection ~\ref{subsectrace}) that the trace of the flux is well-defined when integrated over time intervals. Imposing a geometric condition on the boundary (stronger than just deformable Lipschitz), as well as on the continuity in time of the total mass,  it is shown (Theorem ~\ref{thmweakbalance}) that the resulting flux is \textit{Lipschitz continuous} with respect to boundary deformations and the balance equation is satisfied.
      
      Finally, while this paper is concerned with theoretical aspects of  the balance law formulation, we emphasize its relevance to the numerical simulation of nonlinear hyperbolic conservation laws. More specifically, it serves as a theoretical basis of  finite volume schemes ~\cite{eymard-gallouet,GodlewskiRaviart}; in fact every cell of the discrete mesh is considered as a ``control volume'' in which the balance law is implemented between arbitrary time levels $t_1<t_2.$ The common points between our treatment here and finite volume schemes can be summarized as follows.
      \begin{itemize}
      \item The fact that the integral $\int_{\Omega}u(x,t)dx$ is assumed to be a continuous function of $t$ is very natural when referring to the conserved quantities, such as mass, momentum and energy.
      \item The construction of approximate  fluxes is a primary building block of the finite volume schemes. The fact that the fluxes (evaluated over time intervals) are Lipschitz continuous places them at the position of the ``most regular elements'' in this context. This is in contrast to the complex discontinuities experienced by the flow variables. It therefore makes good sense, from the numerical point-of-view, to aim at approximating these regular fluxes, and then incorporate the approximate fluxes into the balance law.
          
          This is indeed reflected in the GRP methodology ~\cite{BenArtzi-Falcovitz-2003}, the MUSCL-Hancock  ~\cite{MUSCL} scheme, as well as the full plethora of ``Godunov-type'' schemes.

\end{itemize}

\section{\textbf{THE FUNDAMENTAL PRINCIPLE OF THE HYPERBOLIC BALANCE LAW}}\label{secbalancelaw}
As is well known,  the meaning of the $x$ and $t$ derivatives in the conservation equation ~\eqref{eqconslaw} must be clarified since the solutions generate discontinuities, such as shocks or interfaces. The concept of a \textbf{weak solution} is introduced  precisely in order to handle this difficulty  ~\cite[Chapter 11]{evans}, as follows.

     \begin{defn}\label{defweak} The function $u(x,t)$ is a weak solution of ~\eqref{eqconslaw} if the following condition is satisfied:
    for every cylinder $Q=\Omega\times[t_1,t_2]\subseteq \Rn\times\Rplusc,$ if
     $$\phi(x,t)=(\phi_1(x,t),\ldots,\phi_D(x,t)\in C^\infty_0(Q)\bigotimes \RR^D,\,$$ then
        \be\label{eqweaksol}
        \suml_{i=1}^D\int_{t_1}^{t_2}\int_{\Omega}[u_i(x,t)\frac{\pa}{\pa t}\phi_i+ f_i(u(x,t))\cdot \nabla_x\phi_i]dx\,dt=0.
        \ee
      \end{defn}
\subsection{\textbf{BOUNDEDNESS OF THE FLUX DIVERGENCE }}\label{subsecsystemlaw}

       Definition ~\ref{defweak} is a mathematical artifact and does not yield (in a straightforward fashion) the desired balance equality ~\eqref{eqbalancecons}.  The following lemma pretty much summarizes what can be said about the pointwise regularity of the flux function. Observe that in the one-dimensional (spatial) case the lemma already implies the Lipschitz regularity of the flux ~\cite{mathcomp}. Nevertheless, this is not true in the higher dimensional case.

      \begin{lem}\label{lemweakbalance} Let $u(x,t)$ be a weak solution to the system ~\eqref{eqconslaw}, with initial function $u_0\in L^1(\Rn)\cap L^\infty(\Rn)\bigotimes\RR^D.$

      Assume that $u(x,t)$ satisfies the following properties.

      \begin{itemize}\item $u(x,t)$ is locally bounded in $\Rn\times\Rplusc.$

      \item For every fixed bounded $\Omega\subseteq\Rn$ the mass
      \be\label{eqmtcont} m(t)=\int\limits_{\Omega} u(x,t)dx\,\,   \mbox{is a well-defined and continuous function of}\,\, t\in\Rplusc.\ee
      \end{itemize}
      Then
  for every fixed $[t_1,t_2]\subseteq\RR$ the function $g(x;t_1,t_2)=\int_{t_1}^{t_2}f(u(x,t))dt$ satisfies $\nabla_x\cdot g(x;t_1,t_2)\in L^\infty_{loc} (\Rn)\bigotimes \RR^D.$
   \end{lem}


   \begin{proof}For every cylinder $Q=\Omega\times[t_1,t_2]\subseteq \Rn\times\Rplusc$ we define
   \be\label{equxtmaxu0}
      C_Q=\sup\set{ |u(x,t)|, \quad (x,t)\in Q}.
       \ee
    Note that in ~\eqref{eqbalancecons},  the ``fixed time'' integrals in the left-hand side exist by the assumed continuity (in time) of $m(t).$ Pick $\phi(x,t)=\theta(t)\psi(x)$ in Equation ~\eqref{eqweaksol}, where $\theta\in C^\infty_0(t_1,t_2)$ and $\psi\in C^\infty_0(\Omega)\bigotimes\RR^D.$ Take $0\leq\theta\leq 1$ and $\theta(t)=1$ for $t_1+\eps\leq t\leq t_2-\eps.$ Letting $\eps\to 0,$ Equation ~\eqref{eqweaksol} yields
       \be\label{eqintuf}
        \int_{\Omega}[u(x,t_2)-u(x,t_1)]\cdot\psi(x)dx=\int_{\Omega}\int_{t_1}^{t_2}f(u(x,t))dt\cdot\nabla\psi(x)dx.
       \ee
 Equation ~\eqref{eqintuf} can be rewritten as
 $$
          \int_{\Omega}[u(x,t_2)-u(x,t_1)]\cdot\psi(x)dx=\int_{\Omega}g(x;t_1,t_2)\cdot\nabla\psi(x)dx.
$$
Since $|u(x,t)|\leq C_Q$ it follows that
         \be\label{eqestGpsi}
         \Big|\int_{\Omega}g(x;t_1,t_2)\cdot\nabla\psi(x)dx\Big|\leq 2C_Q\|\psi\|_1.
         \ee
 Define the linear functional for $\psi\in C^\infty_0(\Omega)\bigotimes\RR^D$
          $$
          \Gcal\psi=\int_{\Omega}g(x;t_1,t_2)\cdot\nabla\psi(x)dx=\suml_{i=1}^D \int_{\Omega}g_i(x;t_1,t_2)\cdot\nabla\psi_i(x)dx.
          $$
          The estimate ~\eqref{eqestGpsi} shows that $\Gcal$ is continuous with respect to the $L^1$ norm. The density of $C^\infty_0(\Omega)$ in $L^1(\Omega)$ and the $L^1,\,L^\infty$ duality entail that there exists a function $r(x)\in L^\infty(\Omega)\bigotimes\RR^D$ such that
          \be\label{eqdualgr}
           \int_{\Omega}g(x;t_1,t_2)\cdot\nabla\psi(x)dx=\int_{\Omega}r(x)\cdot\psi(x)dx,\quad \psi\in C^\infty_0(\Omega)\bigotimes\RR^D.
          \ee
           We conclude that the distributional divergence of $g(x;t_1,t_2)$ satisfies $\nabla_x\cdot g(x;t_1,t_2)=-r(x)$ in $\Omega.$

           This concludes the proof of the lemma.
       \end{proof}
           \begin{rem}\label{remmtcont} We could replace the continuity
           assumption ~\eqref{eqmtcont} by the stronger assumption that the map $t\to u(\cdot,t)\in L^\infty(\RR)\,\mbox{weak}^\ast$ is continuous. This latter assumption is universally imposed when dealing with entropy solutions to nonlinear conservation laws ~\cite[Section 4.5]{dafermos}. However the continuity condition ~\eqref{eqmtcont} is valid for weak solutions that are not necessarily entropy solutions. In fact, it holds for weak solutions that have bounded (locally in time) total variation. This is expressed by Dafermos as ``mechanism of regularity transfer from the spatial to the temporal variables'' ~\cite[Theorem 4.3.1]{dafermos}.
           \end{rem}

           \subsection{\textbf{TRACES OF FLUXES--GEOMETRIC APPROACH}}\label{subsectrace}

           In order to replace ``weak solutions'' by ``solutions to balance laws'' and make good sense of Equation ~\eqref{eqbalancecons} we need to establish the meaning of \textit{fluxes across domain boundaries.} The regularity result of Lemma ~\ref{lemweakbalance} falls short of this goal. We therefore need to address directly such traces.

           Let $\Omega=\Omega_0\subseteq\Rn$ be a bounded domain with smooth boundary $\Gamma=\Gamma_0=\pa\Omega.$

          Starting with $\Gamma_0$ we can construct a tubular neighborhood ~\cite[Chapter 9, Addendum]{spivak} with the following properties. For some small $0<\delta<1$ there exists family of ``expanding''  smooth bounded domains $\set{\Omega_y\subseteq\Rn,\,\,y\in(-\delta,1-\delta)}$  so that their respective boundaries $\set{\Gamma_y,\,\,y\in(-\delta,1-\delta)}$ form a foliation of a tubular neighborhood of $\Gamma_0.$ The coordinate  $y\in (-\delta,1-\delta)$ is normal to $\Gamma_y$ so that $\frac{\pa}{\pa y}=\nu$ is the unit normal. We designate by $dS_y$ the Lebesgue  surface measure on $\Gamma_y,\,\,y\in(-\delta,1-\delta).$

          In direct continuation to Lemma ~\ref{lemweakbalance} we now have.

          \begin{lem}
          \label{lemtrace} Let $u(x,t)$ be a weak solution to the system ~\eqref{eqconslaw}, with initial function $u_0\in L^1(\Rn)\cap L^\infty(\Rn)\bigotimes\RR^D.$

      Assume that $u(x,t)$ satisfies the following properties.

      \begin{itemize}\item $u(x,t)$ is locally bounded in $\Rn\times\Rplusc.$

      \item For every fixed bounded $\Omega\subseteq\Rn$ the mass
      \be\label{eqmtconta} m(t)=\int\limits_{\Omega} u(x,t)dx\,\,   \mbox{is a well-defined and continuous function of}\,\, t\in\Rplusc.\ee
      \end{itemize}
      For every smooth domain $\Omega$ and the geometric construction above,
      and for every fixed $[t_1,t_2]\subseteq\RR$ define  the trace function $h(y;t_1,t_2)=(h_1(y;t_1,t_2),\ldots,h_D(y;t_1,t_2))$ by
      $$h_i(y;t_1,t_2)=\int_{t_1}^{t_2}\Big[\int_{\Gamma_y}f_i(u(x,t))\cdot\,\nu\, dS_y\Big]dt,\quad i=1,2,\ldots,D,\,\,y\in (-\delta,1-\delta).$$
  Then $h$ is Lipschitz continuous with respect to $y\in(-\delta,1-\delta).$
          \end{lem}
         \begin{proof}
         As in the proof of Lemma ~\ref{lemweakbalance} we obtain (see ~\eqref{eqintuf}) for every smooth domain $\widetilde{\Omega}$
         \be\label{eqintufgena}
        \int_{\widetilde{\Omega}}[u(x,t_2)-u(x,t_1)]\cdot\psi(x)dx=\int_{\widetilde{\Omega}}\int_{t_1}^{t_2}f(u(x,t))dt\cdot\nabla\psi(x)dx.
       \ee
          We construct $\widetilde{\Omega}$ as the tubular domain $$\widetilde{\Omega}=\cup\set{\Gamma_y,\,\,y\in (-\delta,1-\delta)}.$$

          Let $\psi\in C^\infty_0(\widetilde{\Omega})\bigotimes\RR^D$ such that
            \be
            \psi (x)=\theta(y),\quad x\in\Gamma_y,
            \ee
            where $\theta(y)\in C^\infty_0[-\delta,1-\delta)\bigotimes\RR^D.$

            Equation ~\eqref{eqintufgena} can now be rewritten as (where $\theta=(\theta_1,\ldots,\theta_D)$)
            $$
        \int_{\widetilde{\Omega}}[u(x,t_2)-u(x,t_1)]\cdot\psi(x)dx=\suml_{i=1}^D\int_{t_1}^{t_2}\Big[\int_{-\delta}^{1-\delta}\int_{\Gamma_y}f_i(u(x,t))\frac{\pa}{\pa y}\theta_i(y)\cdot\nu\, dS_y dy\Big]dt,
       $$
       namely
       \be\label{eqintufgen}
       \int_{\widetilde{\Omega}}[u(x,t_2)-u(x,t_1)]\cdot\psi(x)dx=\suml_{i=1}^D\int_{-\delta}^{1-\delta}h_i(y;t_1,t_2)\frac{\pa}{\pa y}\theta_i(y)dy.
       \ee
       Define the linear functional
          $$
          \Gcal\theta=\suml_{i=1}^D\int_{-\delta}^{1-\delta}h_i(y;t_1,t_2)\frac{\pa}{\pa y}\theta_i(y)dy,\quad \theta(y)\in C^\infty_0(-\delta,1-\delta)\bigotimes\RR^D.
          $$
          From ~\eqref{eqintufgen} and the boundedness assumption on $u$ we infer that
           $\Gcal$ is continuous with respect to the $L^1(-\delta,1-\delta)$ norm. The density of $C^\infty_0(-\delta,1-\delta)$ in $L^1(-\delta,1-\delta)$ and the $L^1,\,L^\infty$ duality entail that there exists a function $r(y)\in L^\infty(-\delta,1-\delta)\bigotimes\RR^D$ such that
          \be\label{eqdualgrtrace}
           \int_{-\delta}^{1-\delta}h(y;t_1,t_2)\cdot\frac{\pa}{\pa y}\theta(y)dy=\int_{-\delta}^{1-\delta}r(y)\cdot\theta(y)(y)dy,\quad \theta\in C^\infty_0(-\delta,1-\delta)\bigotimes\RR^D.
          \ee
           It follows that the distributional derivative $\frac{\pa}{\pa y}h(y;t_1,t_2)=-r(y)$ is bounded, which concludes the proof of the lemma.
         \end{proof}

         We summarize the above result as the \textit{fundamental theorem of fluxes.}
         \begin{thm}\label{thmweakbalance}
         Let $u(x,t)$ be a weak solution to the system ~\eqref{eqconslaw}, with initial function $u_0\in L^1(\Rn)\cap L^\infty(\Rn)\bigotimes\RR^D.$

      Assume that $u(x,t)$ satisfies the following properties.

      \begin{itemize}\item $u(x,t)$ is locally bounded in $\Rn\times\Rplusc.$

      \item For every fixed bounded $\Omega\subseteq\Rn$ the mass
      \be\label{eqmtcontb} m(t)=\int\limits_{\Omega} u(x,t)dx\,\,   \mbox{is a well-defined and continuous function of}\,\, t\in\Rplusc.\ee
      \end{itemize}
       Then for any smooth bounded domain $\Omega\subseteq\Rn$ and for every time interval $[t_1,t_2]$ the flux $$h_i(t_1,t_2)=\int_{t_1}^{t_2}\int_{\pa\Omega}f_i(u(x,t))\cdot\nu dS\,dt,\quad i=1,2,\ldots,D,$$ is well defined and Equation ~\eqref{eqbalancecons} holds.

         \end{thm}
         \begin{proof}
         In light of Lemma ~\ref{lemtrace} it only remains to establish the validity of the balance equation ~\eqref{eqbalancecons}. Starting from Equation
         ~\eqref{eqintuf} and using the geometric construction above, we select the test function $\psi(x)=(\psi_1(x),\ldots,\psi_D(x))$ as follows.

            $\psi_i(x)\in C^\infty_0(\Omega_0)$ and $\psi_i(x)\equiv 1,\,\,x\in \Omega_{-\delta}.$

            Letting $\delta\to 0$ and using the continuity of the traces obtained in Lemma ~\ref{lemtrace} we obtain ~\eqref{eqbalancecons}.
         \end{proof}

       The statement of Theorem ~\ref{thmweakbalance} is closely related to the more fluid dynamical viewpoint: the ``conservation law'', which is a \textbf{partial differential equation,} is replaced by a \underline{``balance law''}.

      Note that as in the case of weak solutions, no uniqueness assumption is imposed on the solution.

      Theorem ~\ref{thmweakbalance} implies that a weak solution satisfying certain hypotheses (in particular an entropy solution) is a solution to the balance law  in the sense of Definition ~\ref{defnbalance}. It is easy to see that conversely, a solution to the balance law  is a weak solution of the  conservation law ~\eqref{eqconslaw}.

      An important observation is that the flux $h(t_1,t_2)$ is defined \textit{over a time interval.} In other words, there is no meaning attached to the instantaneous value $\int_{\pa\Omega}f_i(u(x,t))\cdot\nu dS.$ However, the flux is continuous with respect to the time interval, as in the following proposition.
      \begin{prop}\label{propcontfluxt} Under the conditions of Theorem
      ~\ref{thmweakbalance}
          the flux $h(t_1,t_2)$ is continuous with respect to $t_1,\,t_2.$
      \end{prop}
      \begin{proof}
        This follows from the balance equation  ~\eqref{eqbalancecons} and the assumption about the continuity of $m(t).$
      \end{proof}
      It is easy to see how to generalize the theorem to bounded domains with piecewise-smooth boundaries. From the point of view of applications, the most important instance is that of polygonal domains. For finite-volume schemes on regular meshes, every cell is  a rectangular box, and we state the result explicitly for this case.
      \begin{cor}
         Let $u(x,t)$ be a weak solution to the system ~\eqref{eqconslaw}, satisfying the conditions of Theorem ~\ref{thmweakbalance}. Let
         $$\Omega=\prod\limits_{i=1}^n[a_i,x_i],$$
           and let $$S_j=\set{y=(y_1,\ldots,y_{j-1},x_j,y_{j+1},\ldots,y_n),\quad y_i\in [a_i,x_i],\,\, i\neq j}$$
           be the  section of $\pa\Omega$ at $x_j.$

           For any $1\leq j\leq n$ and any $0\leq t_1<t_2$ define the flux $$F^j(x_j;t_1,t_2)=\int_{t_1}^{t_2}\int_{S_j}f(u(y,t))\cdot e_jdS_y\,dt\in \RR^D,$$ where $e_j$ is the unit vector in the $x_j$ direction.

           Then $F^j(x_j;t_1,t_2)$ is well defined and indeed is a locally Lipschitz function of $x_j.$ Furthermore, the following balance equation holds.
           \be
           \int_{\Omega}u(x,t_2)dx-\int_{\Omega}u(x,t_1)dx=
     -\suml_{j=1}^n\Big[F^j(x_j;t_1,t_2)-F^j(a_j;t_1,t_2)\Big].
    \ee

      \end{cor}

\end{document}